\documentclass[a4paper,12pt]{amsart}
\usepackage{amssymb}
\usepackage{cite}
\usepackage{ifthen}
\usepackage[dvips]{graphicx}
\nonstopmode \numberwithin{equation}{section}
\usepackage{amssymb}
\usepackage{ifthen}
\usepackage{graphicx}
\usepackage{amsmath}
\usepackage[T1]{fontenc} 
\usepackage[utf8]{inputenc}
\usepackage[usenames,dvipsnames]{color}
\usepackage{color}
\usepackage[english]{babel}
\usepackage{fancyhdr}
\usepackage{fancybox}
\usepackage{tikz}

\theoremstyle{plain}
\newtheorem{prop}{Proposition}

\newtheorem{conj}{Conjecture}

\theoremstyle{definition}

\newtheorem{exm}{Example}[section]
\newtheorem{cor}{Corollary}[section]
\newtheorem{thm}{Theorem}[section]

\newtheorem{lem}{Lemma}[section]
\newtheorem{prob}{Problem}
\newtheorem{rem}{Remark}[section]

\theoremstyle{plain}
\newtheorem*{thmA}{Theorem A}
\newtheorem*{thmB}{Theorem B}
\newtheorem*{thmC}{Theorem C}

\newtheorem*{lemA}{Lemma A}

\newcounter{minutes}
\divide\time by 60
\newcounter{hours}
\multiply\time by 60
\addtocounter{minutes}{-\time}

\newcounter {own}
\def\theown {\thesection       .\arabic{own}}

\newenvironment{pf}[1][]{%
	\vskip 3mm
	\noindent
	\ifthenelse{\equal{#1}{}}%
	{{\slshape Proof. }}%
	{{\slshape #1.} }%
}%
{\qed\bigskip}

\newcounter{alphabet}

\def\be{\begin{equation}}
	\def\ee{\end{equation}}

\newcommand{\bee}{\begin{enumerate}}
	\newcommand{\eee}{\end{enumerate}}

\newcommand{\blem}{\begin{lem}}
	\newcommand{\elem}{\end{lem}}
\newcommand{\bthm}{\begin{thm}}
	\newcommand{\ethm}{\end{thm}}
\newcommand{\bcor}{\begin{cor}}
	\newcommand{\ecor}{\end{cor}}
\newcommand{\beg}{\begin{examp}}
	\newcommand{\eeg}{\end{examp}}
\newcommand{\begs}{\begin{examples}}
	\newcommand{\eegs}{\end{examples}}

\newcommand{\bdefn}{\begin{defn}}
	\newcommand{\edefn}{\end{defn}}

\newcommand{\bprob}{\begin{prob}}
	\newcommand{\eprob}{\end{prob}}
\newcommand{\bei}{\begin{itemize}}
	\newcommand{\eei}{\end{itemize}}

\newcommand{\bcon}{\begin{conj}}
	\newcommand{\econ}{\end{conj}}
\newcommand{\bcons}{\begin{conjs}}
	\newcommand{\econs}{\end{conjs}}
\newcommand{\bprop}{\begin{prop}}
	\newcommand{\eprop}{\end{prop}}
\newcommand{\br}{\begin{rem}}
	\newcommand{\er}{\end{rem}}
\newcommand{\brs}{\begin{rems}}
	\newcommand{\ers}{\end{rems}}
\newcommand{\bo}{\begin{obser}}
	\newcommand{\eo}{\end{obser}}
\newcommand{\bos}{\begin{obsers}}
	\newcommand{\eos}{\end{obsers}}
\newcommand{\bpf}{\begin{pf}}
	\newcommand{\epf}{\end{pf}}
\newcommand{\ba}{\begin{array}}
	\newcommand{\ea}{\end{array}}
\newcommand{\beq}{\begin{eqnarray}}
	\newcommand{\beqq}{\begin{eqnarray*}}
		\newcommand{\eeq}{\end{eqnarray}}
	\newcommand{\eeqq}{\end{eqnarray*}}

\begin{document}

\title{Schwarzian Norm Estimates for Analytic Functions Associated with Convex Functions}

\author{Molla Basir Ahamed$^*$}
\address{Molla Basir Ahamed, Department of Mathematics, Jadavpur University, Kolkata-700032, West Bengal, India.}
\email{mbahamed.math@jadavpuruniversity.in}

\author{Rajesh Hossain}
\address{Rajesh Hossain, Department of Mathematics, Jadavpur University, Kolkata-700032, West Bengal, India.}
\email{rajesh1998hossain@gmail.com}

\author{Sabir Ahammed}
\address{Sabir Ahammed, Department of Mathematics, Jadavpur University, Kolkata-700032, West Bengal, India.}
\email{sabir.math.rs@jadavpuruniversity.in}

\def\thefootnote{}
\footnotetext{ {\tiny File:~\jobname.tex,
printed: \number\year-\number\month-\number\day,
          \thehours.\ifnum\theminutes<10{0}\fi\theminutes }
} \makeatletter\def\thefootnote{\@arabic\c@footnote}\makeatother
\begin{abstract} 
		Let $\mathcal{A}$ denote the class of analytic functions $f$ on the unit disc $\mathbb{D}=\{z\in\mathbb{C}:\;|z|<1\}$ normalized by $f(0)=0$ and $f^{\prime}(0)=1$. In the present article, we consider and $\mathcal{F}(c)$ the subclasses of $\mathcal{A}$ are defined by
	\begin{align*}
		\mathcal{F}(c)=\bigg\{f\in\mathcal{A}:\;{\rm Re}\;\bigg(1+\frac{zf^{\prime\prime}(z)}{f^{\prime}(z)}\bigg)>1-\frac{c}{2},\;\;\mbox{for some}\;c\in(0,3]\bigg\},
	\end{align*}
	and derive sharp bounds for the norms of the Schwarzian and pre-Schwarzian derivatives for functions in and $\mathcal{F}(c)$ expressed in terms of their value $f^{\prime\prime}(0)$, in particular, when the quantity is equal to zero. Moreover, we obtain sharp bounds for distortion and growth theorems for functions in the class $\mathcal{F}(c)$.
\end{abstract}
\maketitle
\pagestyle{myheadings}
\markboth{ M. B. Ahamed, Rajesh Hossain and Sabir Ahammed}{Schwarzian Norm Estimates for Analytic Functions Associated with Convex Functions}
\section{\bf Introduction}
Let \( \mathbb{D} = \{z \in \mathbb{C} : |z| < 1\} \) be the unit disk, and define \( \mathcal{H} \) as the class of analytic functions on \( \mathbb{D} \). The subclass \( \mathcal{LU} \) consists of locally univalent functions, \textit{i.e.,} functions \( f \in \mathcal{H} \) with \( f^{\prime}(z) \neq 0 \) for all \( z \in \mathbb{D} \). For functions $f\in\mathcal{LU}$ defined in a simply connected domain ${\Omega}$, the pre-Schwarzian derivative $P_f$ and the Schwarzian derivative $S_f$ are, respectively, defined by
\begin{align*}
	P_f=\frac{f^{\prime\prime}(z)}{f^{\prime}(z)}\;\mbox{and}\;S_f=(P_f)^{\prime}(z)-\frac{1}{2}(P_f)^2(z)=\frac{f^{\prime\prime\prime}(z)}{f^{\prime\prime}(z)}-\frac{3}{2}\left(\frac{f^{\prime\prime}(z)}{f^{\prime}(z)}\right)^2.
\end{align*}
The pre-Schwarzian and Schwarzian norms are defined by
\begin{align*}
	||P_f||_{\Omega}= \sup_{z\in{\Omega}}|P_f|\eta^{-1}_{\Omega}\;\mbox{and}\;||S_f||_{\Omega}= \sup_{z\in{\Omega}}|S_f|\eta^{-2}_{\Omega}
\end{align*}
respectively, where $\eta_{\Omega}$ is the Poincare density. In particular, if ${\Omega}=\mathbb{D}$, then $||S_f||_{\Omega}$ and $||P_f||_{\Omega}$ are denoted by $||S_f||$ and $||P_f||$,  respectively.\vspace{2mm}

In the following, we discuss some properties of Schwarzian derivatives:
\begin{enumerate}
	\item[$\bullet$] If $\varphi$ is a locally univalent analytic function for which the composition $f\circ \varphi$ is defined, then 
	\begin{align*}
		S_{f\circ \varphi}(z)=S_{f}\circ \varphi(z)\left(\varphi^{\prime}(z)\right)+S_{\varphi}(z).
	\end{align*}
	\item[$\bullet$] The Schwarzian derivative is invariant under M\"obius transformation, \emph{i.e.,} $S_{T\circ\varphi}=S_f$ for any M\"obius transformation $T$ of the form 
	\begin{align*}
		T(z)=\frac{az+b}{cz+d},\; ad-bc\neq 0,\; a, b, c, d\in\mathbb{C}.
	\end{align*}
	\item[$\bullet$] It is easy to verify that $S_f(z)=0$ if, and only, if $f$ is a M\"obius transformation. \vspace{1.2mm}
	\item[$\bullet$] There is a classical relation between the Schwarzian derivative and second order linear differential equations. If $S_f=2p$ and $u=\left(f^{\prime}\right)^{-1/2}$, then 
	\begin{align*}
		u^{\prime\prime}+pu=0.
	\end{align*}
	Conversely, if $u_1$, $u_2$ are linearly independent solutions of this D.E. and $f=u_1/u_2$, then $S_f=2p$.  
\end{enumerate}
The pre-Schwarzian and Schwarzian derivatives are key tools in geometric function theory, particularly for characterizing Teichmüller space through embedding models. They also play a crucial role in studying the inner radius of univalency for planar domains and quasiconformal extensions \cite{Lehto-1987,Lehto-JAM-1979}. Their study dates back to Kummer $(1836)$, who introduced the Schwarzian derivative in the context of hypergeometric $PDEs$. Since then, extensive research has explored their connections to univalent functions, leading to several sufficient conditions for univalency.\vspace{2mm}

Let $\mathcal{A}$ denote the subclass of $\mathcal{H}$, a class of analytic functions on the unit disk $\mathbb{D}$,  consisting of functions $f$ with normalized conditions $f(0)=f^{\prime}(z)-1=0$. Thus, any function $f$ in $\mathcal{A}$ has the Taylor series expansion of the form 
\begin{align}\label{Eq-1.3}
	f(z)=z+\sum_{n=2}^{\infty}a_nz^n\;\; \mbox{for all}\;\; z\in\mathbb{D}.
\end{align}
Let $\mathcal{S}$ be the subclass of $\mathcal{A}$ consisting of univalent (that is, one-to-one) functions. A function $f\in \mathcal{A}$ is called starlike (with respect to the origin) if $f(\mathbb{D})$ is starlike with respect to the origin and convex if $f(\mathbb{D})$ is convex. The class of all univalent starlike (resp. convex) functions in $\mathcal{A}$ is denoted by $\mathcal{S}^*$ (resp. $\mathcal{C}$). However, it is well-known that a function $f\in\mathcal{S}^*$ (resp. $f\in\mathcal{C}$) if, and only if, 
\begin{align*}
	{\rm Re}\left(\frac{zf^{\prime}(z)}{f(z)}\right)>0\; \left(\mbox{resp.}\; {\rm Re}\left(1+\frac{zf^{\prime\prime}(z)}{f^{\prime}(z)}\right)>0\right),\; z\in\mathbb{D}.
\end{align*}
The following observations are important for starlike and convex functions.
\begin{enumerate}
	\item[$\bullet$] The characterizations of starlikeness or convexity are sufficient but not necessary for univalency.\vspace{2mm}
	
	\item[$\bullet$] {\bf Nehari's criteria.} Nehari developed univalency involving Schwarzian derivative, where sufficient condition is almost necessary in the sense that scalar terms vary.
\end{enumerate}
The next two results provide necessary and sufficient criteria for a function to be univalent.
\begin{thmA}\emph{(Kraus-Nehari's Theorem) (Necessary condition)}
	Let $f$ be a univalent function. Then $f$ satisfies
	\begin{align*}
		|S_f(z)|\leq\frac{6}{(1-|z|^2)^2}\; \mbox{for}\; z\in\mathbb{D}.
	\end{align*}
	Moreover, the constant $6$ cannot be replaced by a smaller one.
\end{thmA}
\begin{thmB}\emph{(Nehari's Theorem) (Sufficient condition)}
	Let $f$ be a locally univalent function. If $f$ satisfies
	\begin{align*}
		|S_f(z)|\leq\frac{2}{(1-|z|^2)^2}\; \mbox{for}\; z\in\mathbb{D},
	\end{align*}
	then $f$ is univalent in $\mathbb{D}$. Moreover, the constant $2$ cannot be replaced by a larger one.
\end{thmB}
This was Nehari's motivation to study Schawarzian derivatives as well as Schawarzian norm. It is well-known that the pre-Schwarzian norm $||Pf||\leq 6$ holds for the univalent analytic function $f$ is defined in $\mathbb{D}$. In $1972$, Becker \cite{Becker-JRAM-1983} used the pre-Schwarzian derivative to obtain the sufficient condition that the function in $\mathbb{D}$ is univalent, in other words, if $||Pf||\leq1$, then the function $f$ is univalent in $\mathbb{D}$. In 1976, Yamashita \cite{Yamashita-MM-1976} proved that $||Pf||$ is finite if, and only if, $f$ is uniformly locally univalent in $\mathbb{D}$, \emph{i.e.}, there exists a constant $\rho$ such that $f$ is univalent on the hyperbolic disk $|(z-a)/(1-\bar{a}z)|<\tanh\rho$ of radius $\rho$ for every $a\in\mathbb{D}$. Sugawa \cite{Sugawa-AUMCDS-1996} studied and established the norm of the pre-Schwarzian derivative of the strongly starlike functions of order $\alpha\; (0<\alpha\leq1)$. Yamashita\cite{Yamashita-HMJ-1999} generalized sugawa's results by a general class named Gelfer-starlike of exponential order $\alpha (\alpha>0)$ and the Gelfer-close-to-convex of exponential order $(\alpha,\beta)$ ($\alpha>0$, $\beta>0$). These Gelfer classes also contain the classical starlike, convex, close-to-convex all of order $\alpha$ ($0\leq\alpha<1$), which are denote by $\mathcal{S^*(\alpha)}$, $\mathcal{C(\alpha)}$, $\mathcal{K(\alpha)}$ respectively, and so on.\vspace{2mm}

Here, we recall that a function $f\in\mathcal{A}$ is called close-to-convex if $f(\mathbb{D})$ in $\mathbb{C}$ is the union of closed half lines with pairwise disjoint interiors. However, in \cite{Okuyama-CVTA-2000}, Okuyama  studied the subclass of $\alpha$-spirallike functions of order ($-\pi/2<\alpha<\pi/2$), and later a general class call $\alpha$-spirallike functions of order $\rho$ $(0\leq\rho<1)$ considered by Aghalary and Orouji \cite{Aghalary-Orouji-COAT-2014}. Recently, Ali and Pal \cite{Ali-Pal-MM-2023} studied the sharp estimate of the pre-Schwarzian norm for the Janowski starlike functions. Other subclasses have also been widely studied, such as meromorphic function exterior of the unit disk \cite{Ponnusamy-Sugawa-JKMS-2008}, subclass of strong starlike function \cite{Ponnusamy-Sahoo-M-2008}, uniformly convex and uniformly starlike function \cite{Kanas-AMC-2009} and bi-univalent function \cite{Rahmatan-Najafzadeh-Ebadian-BIMS-2017}. For the pre-Schwarzian norm estimates of other function forms such as convolution operator and integral operator, we refer to the articles \cite{Choi-Kim-Ponnusamy-Sugawa-JMAP-2005,Kim-Sugawa-PEMS-2006,Parvatham-Ponnusamy-Sahoo-HMJ-2008,Ponnusamy-Sahoo-JMAA-2008,Nehari-BAMS-1949} and references therein. The pioneering work on the bound \( ||Sf|| \leq 6 \) for a univalent function \( f \in \mathcal{A} \) was first introduced by Kraus \cite{Kraus-1932} and later revisited by Nehari \cite{Nehari-BAMS-1949}. In the same paper, Nehari also proved that if $||Sf||\leq2$, then the function $f$ is univalent in $\mathbb{D}$.\vspace{2mm}

The Schwarzian norm plays a significant role in the theory of quasiconformal mappings and Teichm\"uller space (see \cite{Lehto-1987}). A mapping \( f : \hat{\mathbb{C}} \to \hat{\mathbb{C}} \) of the Riemann sphere \( \hat{\mathbb{C}} := \mathbb{C} \cup \{\infty\} \) is said to be a \( k \)-quasiconformal (\( 0 \leq k < 1 \)) mapping if it is a sense-preserving homeomorphism of \( \hat{\mathbb{C}} \) and has locally integrable partial derivatives on \( \mathbb{C} \setminus \{f^{-1}(\infty)\} \), satisfying \( |f_{\bar{z}}| \leq k |f_z| \) almost everywhere.  On the other hand, Teichm\"uller space \( \mathcal{T} \) can be identified with the set of Schwarzian derivatives of analytic and univalent functions on \( \mathbb{D} \) that have quasiconformal extensions to \( \hat{\mathbb{C}} \). It is known that \( \mathcal{T} \) is a bounded domain in the Banach space of analytic functions on \( \mathbb{D} \) with a finite hyperbolic sup-norm (see \cite{Lehto-1987}).  \vspace{2mm}

The Schwarzian derivative and quasiconformal mappings are connected through key results presented below.
\begin{thmC}\cite{Ahlfrors-Weill-PAMS-2012,Kühnau-MN-1971}
	If $f$ extends to a $k$-quasiconformal $(0\leq k<1)$ mapping of the Riemann share $\hat{\mathbb{C}}$, then $||S_f||\leq 6k$. Conversely, if $||S_f||\leq 2k$, then $f$ extends to a $k$-quasiconformal mapping of the Riemann sphere $\hat{\mathbb{C}}$.
\end{thmC}
Regarding to the estimates of the Schwarzian norm for the subclasses of univalent functions \emph{i.e.,} of functions $f$ that satisfy:
\begin{align*}
	\bigg|\arg\left(\frac{zf^{\prime}(z)}{f(z)}\right)\bigg|<\alpha\frac{\pi}{2},\; z\in\mathbb{D},
\end{align*}
where $0\leq \alpha<1$. In $1996$, Suita \cite{Suita-JHUED-1996} studied the class $\mathcal{C(\alpha)}$, $0\leq \alpha<1$ and using the integral representation of functions in $\mathcal{C}(\alpha)$ proved that the Schwarzian norm satisfies the sharp inequality 
\begin{align*}
	||S_f||\leq\begin{cases}
		2,\;\;\;\;\;\;\;\;\;\;\;\;\;\;\;\;\; \mbox{if}\; 0\leq \alpha\leq 1/2,\\
		8\alpha(1-\alpha),\;\;\;\;\mbox{if}\; 1/2<\alpha<1.
	\end{cases}
\end{align*}
For a constant $\beta\in (-\pi/2, \pi/2)$, a function $f\in\mathcal{A}$ is called $\beta$-spiral like if $f$ is univalent on $\mathbb{D}$ and for any $z\in\mathbb{D}$, the $\beta$-logarithmic spiral $\{f(z)\exp\left(-e^{i\beta}t\right);\; t\geq 0\}$ is contained in $f(\mathbb{D})$. It is equivalent to the condition that ${\rm Re} \left(e^{-i\beta}zf^{\prime}(z)/f(z)\right)>0$ in $\mathbb{D}$ and we denote by $\mathcal{SP}(\beta)$, the set of all $\beta$-spiral like functions. Okuyama \cite{Okuyama-CVTA-2000} give the best possible estimate of the norm of pre-Schwarzian derivatives for the class $\mathcal{SP}(\beta)$.\vspace{2mm}

A function $f\in\mathcal{A}$ is said to be uniformly convex function if every circular arc (positively oriented) of the form $\{z\in\mathbb{D} : |z-\eta|=r\}$, $\eta\in\mathbb{D}$, $0<r<|\eta|+1$ is mapped by $f$ univalently onto a convex arc. The class of all uniformly convex functions is denoted by $\mathcal{UCV}$. In particular, $\mathcal{UCV}\subset \mathcal{K}$. It is well-known that (see \cite{Goodman-APM-1991}) a function $f\in\mathcal{A}$ is uniformly convex if, and only if, 
\begin{align*}
	{\rm Re}\left(1+\frac{z f^{\prime\prime}(z)}{f^{\prime}(z)}\right)>\bigg|\frac{z f^{\prime\prime}(z)}{f^{\prime}(z)}\bigg|^2\; \mbox{for}\; z\in\mathbb{D}.
\end{align*}
In \cite{Kanas-Sugawa-APM-2011}, Kanas and Sugawa established that the Schwarzian norm satisfies \( ||S_f|| \leq 8/\pi^2 \) for all \( f \in \mathcal{UCV} \), with the bound being sharp.  Recently, Schwarzian norm estimates for other subclasses of univalent functions have been gradually studied by many people, such as concave function class \cite{Bhowmik-Wriths-CM-2012}, Robertson class \cite{Ali-Pal-BDS-2023} and other univalent analytic subclasses [6]. Therefore, by using the pre-Schwarzian  and Schwarzian norms to study the univalence and quasiconformal extension problems of analytic function arouse a new wave of research interest.\vspace{2mm}

Our research builds upon a foundational body of work in geometric function theory. The concept of the Schwarzian derivative with $f^{\prime\prime}(0)=0$ for convex mappings of order was initially introduced by Carrasco and Hern\'andez \cite{Carrasco-Hernández-AMP-2023}. Subsequently, Wang \cite{Wang-Li-Fan-MM-2024} expanded on this by providing pre-Schwarzian and Schwarzian norm estimates for subclasses of univalent functions. Most recently, Ahamed and Hossain \cite{Ahamed-Hossain-2024} further developed these ideas by establishing pre-Schwarzian and Schwarzian norm estimates for the class of Ozaki close-to-convex functions. This progression highlights the ongoing significance and evolution of research in this area.\vspace*{2mm}


In \cite{Chuaqui-Duren-Osgood-AASFM-2011}, Chuaqui \emph{et. al.} proved a result by applying the Schwarz-Pick lemma and the fact that the expression $1+z(f^{\prime\prime}/f^{\prime})(z)$ is subordinate to the half-plan mapping $\ell(z)=(1+z)/(1-z)$, which is 
\begin{align}\label{Eq-2.2}
	1+\frac{zf^{\prime\prime}(z)}{f^{\prime}(z)}=\ell(w(z))=\frac{1+w(z)}{1-w(z)}
\end{align}
for some function $w : \mathbb{D}\to\mathbb{D}$ holomorphic and such that $w(0)=0$. \vspace{2mm}

The expression which is defined in \eqref{Eq-2.2} allowed us to obtain other characterizations for the convex functions: 
\begin{align}\label{Eq-22.3}
	f\in\mathcal{C}\; \mbox{if, and only if,}\; {\rm Re}\left(1+\frac{zf^{\prime\prime}(z)}{f^{\prime}(z)} \right)\geq \frac{1}{4}\left( 1-|z|^2\right)\bigg| \frac{f^{\prime\prime}(z)}{f^{\prime}(z)}  \bigg|^2,
\end{align}
and 
\begin{align}\label{Eq-22.4}
	f\in\mathcal{C}\; \mbox{if, and only if,}\; \bigg|\left(1-|z|^2\right)\frac{f^{\prime\prime}(z)}{f^{\prime}(z)} -2\bar{z} \bigg|\leq 2,
\end{align}
for all $z\in\mathbb{D}$.\vspace{2mm}

Let $\mathcal{P}$ denote the class of analytic functions $p$ in the unit disc $\mathbb{D}$ such that $p(0) = 1$ and $\operatorname{Re} p(z) > 0$ in $\mathbb{D}$. The class $\mathcal{P}$ is known as the well-known \textit{Carathéodory} class. Let $\varphi \in \mathcal{P}$ be a univalent function with $\varphi'(0) > 0$ for all $z \in \mathbb{D}$, and suppose that the image domain $\varphi(\mathbb{D})$ is symmetric with respect to the real axis and starlike with respect to $1$. The following unified families of univalent, convex, and starlike functions have been
\begin{align*}
	\mathcal{C}(\phi):=\bigg\{f\in\mathcal{S}:\;1+\frac{zf^{\prime\prime}(z)}{f^{\prime}(z)}\prec\phi(z)\;\;\;\mbox{for}\;\;z\in\mathbb{D}\bigg\}.
\end{align*}
and
\begin{align*}
	\mathcal{S^*}(\phi):=\bigg\{f\in\mathcal{S}:\;\frac{zf^{\prime}(z)}{f(z)}\prec\phi(z)\;\;\;\mbox{for}\;\;z\in\mathbb{D}\bigg\}.
\end{align*}
respectively, introduced and investigated by Ma and Minda \cite{Ma-Minda-1992}. For instance, when $\varphi(z) = \frac{1 + z}{1 - z}$, then $\mathcal{C}(\varphi)$ and $\mathcal{S}^*(\varphi)$ are the families $\mathcal{C}$ and $\mathcal{S}^*$ of convex and starlike functions, respectively. Define the class     $\mathcal{C}\left(1+\frac{cz}{1-z}\right)=:\mathcal{F}(c)$ for some $c\in(0,3]$, \emph{i.e.,} the family $\mathcal{F}(c)$ is defined by
\begin{align*}
	\mathcal{F}(c)=\bigg\{f\in\mathcal{A}:\;{\rm Re}\left(1+\frac{zf^{\prime\prime}(z)}{f^{\prime}(z)}\right)>1-\frac{c}{2}\;\;\;\mbox{for}\;\;z\in\mathbb{D}\bigg\}.
\end{align*}
In terms of subordination, the class $\mathcal{F}(c)$ can be defined as:
\begin{align}\label{Eq-1.5}
	f\in\mathcal{F}(c)\iff{\rm Re  }\left(1+\frac{zf^{\prime\prime}(z)}{f^{\prime}(z)}\right)\prec1+\frac{cz}{1-z}.
\end{align}
The class $\mathcal{F}(c)$ has been considered in the literature for various study. For example, in \cite{Allu-Sharma-BDS-2024}, Allu and Sharma have extensively explored some geometric properties of the function class $\mathcal{F}(c)$. In \cite{Pon-Sha-Wirth-JAMS-2020}, Ponnusamy \emph{et. al.} studied logarithmic coefficient problems for the class $\mathcal{F}(c)$ ans established many significant results. \vspace{2mm}
 
 In this paper, we present several key findings in geometric function theory for the class $\mathcal{F}(c)$ and particularly for the class $\mathcal{F}(2)=\mathcal{C}$ of convex functions. Theorem \ref{Th-2.1} establishes an equivalent relation and a crucial bound essential for determining the radius of concavity, exemplified by a case achieving equality. Theorem \ref{Th-2.2} derives growth and distortion theorems, accompanied by two sharp corollaries. We further investigate a specific class of functions in Theorem \ref{Th-2.3}, determining the best possible bound for the pre-Schwarzian derivative. Theorem \ref{Th-2.4} provides the sharp Schwarzian derivative. Finally, in Theorem \ref{Th-2.5}, we establish an important bound with significant implications for future research.
\section{\bf{Pre-Schwarzian, Schwarzian norm Estimates, and other properties of the functions in the class $\mathcal{F}(c)$}}
In this section, we first give the equivalent characterization of \eqref{Eq-22.3} and \eqref{Eq-22.4} for the class $\mathcal{F}(c)$ and obtain Theorem \ref{Th-2.1}. Next, we present the distortion and growth theorem (Theorem \ref{Th-2.2}) and derive the pre-Schwarzian and Schwarzian norms for functions $f$ in the class  $\mathcal{F}(c)$ in terms of $f^{\prime\prime}(0)$. Our results generalize to the classical class $\mathcal{C}$ of convex functions, as shown through corollaries. 
\begin{thm}\label{Th-2.1} For $c\in(0,3]$, the following are equivalent:
	\begin{enumerate}
		\item[{(i)}] $f\in\mathcal{F}(c)$.\vspace{2mm}
		
		\item[{(ii)}] $\displaystyle
			{\rm Re  }\left(1+\frac{zf^{\prime\prime}(z)}{f^{\prime}(z)}\right)\geq1-\frac{c}{2}+\left(\frac{1-|z|^2}{2c}\right)\bigg|\frac{zf^{\prime\prime}(z)}{f^{\prime}(z)}\bigg|^2$.\vspace{2mm}
		
		\item[{(iii)}] $\displaystyle\bigg|	(1-|z|^2)\left(\frac{f^{\prime\prime}(z)}{f^{\prime}(z)}\right)-c\bar{z}\bigg|\leq c.
		$
	\end{enumerate} The inequalities {(ii)} and {(iii)} both are sharp for the function
	\begin{align}\label{Eq-22.11}
		f_{c}(z)=\frac{(1-z)^{1-c}-1}{c-1}\; \mbox{for}\; z\in\mathbb{D}\; \mbox{with}\; c\in(0,3].
	\end{align}
\end{thm}
As a consequence of Theorem \ref{Th-2.1}, we obtain two results for functions in the class $\mathcal{F}(2)$ involving sharp inequalities.
\begin{cor}\label{Cor-2.1}
	If $f\in \mathcal{C} := \mathcal{F}(2)$ be of the form \eqref{Eq-1.3}, then we have
	\begin{align}\label{Eq-22.1}
		\bigg|	(1-|z|^2)\left(\frac{f^{\prime\prime}(z)}{f^{\prime}(z)}\right)-2\bar{z}\bigg|\leq 2.
	\end{align}
	The inequality \eqref{Eq-22.1} is sharp. 
\end{cor}
\begin{cor}\label{Cor-2.2}
	If $f\in \mathcal{C} := \mathcal{F}(2)$ be of the form \eqref{Eq-1.3}, then we have
	\begin{align}\label{Eq-22.2}
		{\rm Re  }\left(1+\frac{zf^{\prime\prime}(z)}{f^{\prime}(z)}\right)\geq\left(\frac{1-|z|^2}{4}\right)\bigg|\frac{zf^{\prime\prime}(z)}{f^{\prime}(z)}\bigg|^2.
	\end{align}
	The inequality \eqref{Eq-22.2} is sharp. 
\end{cor}
\begin{exm}
	For the sharpness of the inequalities \eqref{Eq-22.1} and \eqref{Eq-22.2}, we consider the function defined in \eqref{Eq-22.11} with $c=2$ as 
	\begin{align*}
		f_2(z)=\frac{1}{(1-z)}-1
	\end{align*}
	A simple computation using \eqref{Eq-22.11} shows that
	\begin{align*}
		1+\frac{zf_2^{\prime\prime}(z)}{f_2^{\prime}(z)}=\frac{1+z}{1-z}.
	\end{align*}
	Moreover, we have 
	\begin{align*}
		{\rm 	Re}\left(1+\frac{zf_2^{\prime\prime}(z)}{f_2^{\prime}(z)}\right)>0
	\end{align*}
	hence, it is clear that $f_2\in\mathcal{F}(2)$.\vspace{2mm}
	
	To show the inequality \eqref{Eq-22.1} of Corollary \ref{Cor-2.1} is sharp, we consider $z=r<1$ and establish that 
	\begin{align*}
		\bigg|	(1-|z|^2)\left(\frac{f_2^{\prime\prime}(z)}{f_2^{\prime}(z)}\right)-2\bar{z}\bigg|=\bigg|	(1-r^2)\left(\frac{2}{1-r}\right)-2r\bigg|=|2-2r+2r|=2.
	\end{align*}
	
	\noindent To show the inequality \eqref{Eq-22.2} in Corollary \ref{Cor-2.2} is sharp, we see from \eqref{Eq-2.3} (Proof of Theorem \ref{Th-2.1}) that 
	\begin{align}\label{Eq-22.44}
		\phi(z)=\frac{\frac{f_2^{\prime\prime}(z)}{f_2^{\prime}(z)}}{\frac{zf_2^{\prime\prime}(z)}{f_2^{\prime}(z)}-1}=1.
	\end{align}
	Thus, it is clear that $|\phi(z)|^2=1$ which further leads to 
	\begin{align*}
		{\rm Re  }\left(1+\frac{zf_2^{\prime\prime}(z)}{f_2^{\prime}(z)}\right)=\left(\frac{1-|z|^2}{4}\right)\bigg|\frac{zf_2^{\prime\prime}(z)}{f_2^{\prime}(z)}\bigg|^2.
	\end{align*}
	
\end{exm}
\begin{proof}[\bf Proof of Theorem \ref{Th-2.1}]
	First, we will prove that $(i)$ is equivalent to $(ii)$. For $\beta>0$, let  $f\in\mathcal{F}(c)$ be of the form \eqref{Eq-1.3}. Then from \eqref{Eq-1.5}, we have
	\begin{align}\label{Eq-4.1}
		1+\frac{zf^{\prime\prime}(z)} {f^{\prime}(z)}\prec1+\frac{cz}{1-z}.
	\end{align}
	\noindent Hence, then there exists an analytic function $\omega: \mathbb{D}\rightarrow\mathbb{D}$ with $\omega(0)=0$ such that
	\begin{align*}
		1+\frac{zf^{\prime\prime}(z)} {f^{\prime}(z)}=1+\frac{c\omega(z)}{1-\omega(z)}.
	\end{align*}
	Let $\omega(z)=z\phi(z)$ for some analytic function $\phi$ that satisfy $\phi(\mathbb{D})\subseteq\mathbb{D}$. Then it follows from \eqref{Eq-4.1} that
	\begin{align*}
		\frac{f^{\prime\prime}(z)}{f^{\prime}(z)}=\frac{c\omega(z)}{z(1-\omega(z))}
	\end{align*}
	which yields that
	\begin{align}\label{Eq-2.3}
		\phi(z)=\frac{\frac{f^{\prime\prime}(z)}{f^{\prime}(z)}}{\frac{zf^{\prime\prime}(z)}{f^{\prime}(z)}+c}.
	\end{align}
	Since $|\phi(z)|^2\leq1$, an easy computation leads to
	\begin{align}\label{Eq-2.4}
		\bigg|\frac{f^{\prime\prime}(z)}{f^{\prime}(z)}\bigg|^2\leq c^2+2c	{\rm Re}\left(\frac{zf^{\prime\prime}(z)}{f^{\prime}(z)}\right)+|z|^2\bigg|\frac{f^{\prime\prime}(z)}{f^{\prime}(z)}\bigg|^2.
	\end{align}
	By factorizing, we obtain
	\begin{align*}
		\nonumber c\left[c+2{\rm Re}\left(\frac{zf^{\prime\prime}(z)}{f^{\prime}(z)}\right)\right]\ge(1-|z|^2)\bigg|\frac{f^{\prime\prime}(z)}{f^{\prime}(z)}\bigg|^2
	\end{align*}
	which implies that
	\begin{align}
		{\rm Re}\left(\frac{zf^{\prime\prime}(z)}{f^{\prime}(z)}\right)\geq-\frac{c}{2}+\left(\frac{1-|z|^2}{2c}\right)\bigg|\frac{f^{\prime\prime}(z)}{f^{\prime}(z)}\bigg|^2.
	\end{align}
	Consequently, we have the desired inequality
	\begin{align*}
		{\rm Re  }\left(1+\frac{zf^{\prime\prime}(z)}{f^{\prime}(z)}\right)\leq1-\frac{c}{2}+\left(\frac{1-|z|^2}{2c}\right)\bigg|\frac{zf^{\prime\prime}(z)}{f^{\prime}(z)}\bigg|^2.
	\end{align*}
	
	\noindent Next, we will prove that $(ii)$ is equivalent to $(iii)$.\vspace{2mm} 
	
	\noindent Multiplying \eqref{Eq-2.4} by $(1-|z|^2)$ both side, we have
	\begin{align*}
		(1-|z|^2)^2\bigg|\frac{f^{\prime\prime}(z)}{f^{\prime}(z)}\bigg|^2\leq c^2(1-|z|^2)+2c(1-|z|^2){\rm Re}\left(\frac{zf^{\prime\prime}(z)}{f^{\prime}(z)}\right)
	\end{align*}
	which implies that 
	\begin{align*}
		(1-|z|^2)^2\bigg|\frac{f^{\prime\prime}(z)}{f^{\prime}(z)}\bigg|^2-2c(1-|z|^2){\rm Re}\left(\frac{zf^{\prime\prime}(z)}{f^{\prime}(z)}\right)+c^2|z|^2\leq c^2.
	\end{align*}
	Hence, the required inequality is established.
	\begin{align*}
		\bigg|	(1-|z|^2)\left(\frac{f^{\prime\prime}(z)}{f^{\prime}(z)}\right)-c\bar{z}\bigg|\leq c.
	\end{align*}
	Considering the function \( f_c \) defined in \eqref{Eq-22.11}, we establish the sharpness of the above two inequalities, as confirmed in Corollary \ref{Cor-2.1} and Corollary \ref{Cor-2.2} for the spacial case \( c= 2 \). This completes the proof.
\end{proof}
In the next result, with the additional condition that $f^{\prime\prime}(0)=0$, $\phi$ satisfies that $\phi(z)=z\psi(z)$. for some analytic function $\psi$ with $|\psi(z)|<1$. We say that $f\in \mathcal{F}^0(c)$ if $f\in\mathcal{F}(c)$ and $f^{\prime\prime}(0)=0$. In this way, we give our results as follows.
\begin{thm}\label{Th-2.2}
	For $c\in(0,3]$, let $f\in \mathcal{F}^0(c)$ be of the form \eqref{Eq-1.3}. Then we have
	\begin{align}
		\frac{1}{(1+|z|^2)^{-\frac{c}{2}}}\leq|f^{\prime}(z)|\leq\frac{1}{(1-|z|^2)^\frac{c}{2}}	
	\end{align}
	and
	\begin{align}
		\int_{0}^{|z|}\frac{1}{(1+\xi^2)^{-\frac{c}{2}}} d|\xi|\leq	|f(z)|\leq\int_{0}^{|z|}\frac{1}{(1-\xi^2)^\frac{c}{2}} d|\xi|.
	\end{align}
	All of these estimates are sharp. Equality holds for the function 
	\begin{align*}
		f_{c, \lambda}(z)=\int_{0}^{|z|}\frac{1}{(1-\lambda \zeta^2)^{c/2}}d|\zeta|
	\end{align*} 
	for some $\lambda\in\mathbb{C}$ with $|\lambda|=1$.
\end{thm}
As a consequence of Theorem \ref{Th-2.2}, we obtain the following result.
\begin{cor}
	If $f\in \mathcal{C} := \mathcal{F}^0(2)$ be of the form \eqref{Eq-1.3}, then the sharp inequality 
	\begin{align}
		\frac{1}{(1+|z|^2)}\leq|f^{\prime}(z)|\leq\frac{1}{(1-|z|^2)}	
	\end{align}
	and
	\begin{align}
		\int_{0}^{|z|}\frac{1}{(1+\xi^2)} d|\xi|\leq	|f(z)|\leq\int_{0}^{|z|}\frac{1}{(1-\xi^2)} d|\xi|.
	\end{align}
	All of these estimates are sharp. Equality holds for the function \begin{align*}
		f_{2, \lambda}(z)=\int_{0}^{|z|}\frac{1}{(1-\lambda \zeta^2)}d|\zeta|
	\end{align*} for some $\lambda\in\mathbb{C}$ with $|\lambda|=1$.
\end{cor}
\begin{proof}[\bf Proof of Theorem \ref{Th-2.2}]
	Let $f\in\mathcal{F}^0(c)$, and from \eqref{Eq-2.3}, we obtain that $\phi(0)=0$. Then  by using the Schwarz lemma, we get
	\begin{align*}
		\bigg|\frac{\frac{f^{\prime\prime}(z)}{f^{\prime}(z)}}{\frac{zf^{\prime\prime}(z)}{f^{\prime}(z)}+c}\bigg|^2\leq|z|^2
	\end{align*}
	which is equivalent to the inequality
	\begin{align*}
		\bigg|\frac{f^{\prime\prime}(z)}{f^{\prime}(z)}\bigg|^2\leq c^2|z|^2+2c|z|^2{\rm Re}\left(\frac{zf^{\prime\prime}(z)}{f^{\prime}(z)}\right)+|z|^2\bigg|\frac{zf^{\prime\prime}(z)}{f^{\prime}(z)}\bigg|^2.
	\end{align*}
	Thus, we have the estimate
	\begin{align}
		(1-|z|^4)\bigg|\frac{f^{\prime\prime}(z)}{f^{\prime}(z)}\bigg|^2\leq c^2|z|^2+2c|z|^2{\rm Re}\left(\frac{zf^{\prime\prime}(z)}{f^{\prime}(z)}\right).
	\end{align}
	Multiplying both sides the above inequality by $(1-|z|^4)$, we obtain 
	\begin{align*}
		(1-|z|^4)^2&\bigg|\frac{f^{\prime\prime}(z)}{f^{\prime}(z)}\bigg|^2-2c|z|^2(1-|z|^4){\rm Re}\left(\frac{zf^{\prime\prime}(z)}{f^{\prime}(z)}\right)\\&\leq c^2|z|^2(1-|z|^4).
	\end{align*}
	Adding \( \left(c|z|^2 \bar{z} \right)^2 \) to both sides of the above inequality leads to
	\begin{align*}
		(1-|z|^4)^2&\bigg|\frac{f^{\prime\prime}(z)}{f^{\prime}(z)}\bigg|^2-2c|z|^2(1-|z|^4){\rm Re}\left(\frac{zf^{\prime\prime}(z)}{f^{\prime}(z)}\right)+c^2|z|^4|\bar{z}|^2\\&\leq c^2|z|^2(1-|z|^4)+c^2|z|^4|\bar{z}|^2\nonumber.
	\end{align*}
	Multiplying both side by $|z|$, then by simple calculation, we obtain
	\begin{align*}
		\bigg|(1-|z|^4)\frac{zf^{\prime\prime}(z)}{f^{\prime}(z)}-c|z|^4\bigg|\leq c|z|^2
	\end{align*}
	which implies that
	\begin{align*}
		\frac{c|z|^2}{1+|z|^2}\leq{\rm Re}\left(\frac{zf^{\prime\prime}(z)}{f^{\prime}(z)}\right)\leq\frac{c|z|^2}{1-|z|^2}.
	\end{align*}
	Let $z=re^{i\theta}$. Then it is easy to see that
	\begin{align*}
		\frac{c r}{1+r^2}\leq\dfrac{\partial}{\partial r}\left(\log|f^{\prime}(re^{i\theta})|\right)\leq\frac{c r}{1-r^2}.
	\end{align*}
	Integrating the above estimate w.r.t. $r$, we obtain
	\begin{align*}
		\frac{1}{(1+|z|^2)^{-\frac{c}{2}}}\leq|f^{\prime}(z)|\leq\frac{1}{(1-|z|^2)^\frac{c}{2}}.
	\end{align*}
	Next, the growth part of the theorem follows from the upper bound
	\begin{align*}
		|f^{\prime}(re^{i\theta})|=\bigg|\int_{0}^{r}f^{\prime}(re^{i\theta})e^{i\theta} dt\bigg|\leq\int_{0}^{r}|f^{\prime}(re^{i\theta})| dt\leq\int_{0}^{r}\frac{1}{(1-t^2)^\frac{c}{2}} dt
	\end{align*}
	which implies that
	\begin{align*}
		|f(z)|\leq\int_{0}^{|z|}\frac{1}{(1-\xi^2)^\frac{c}{2}} d|\xi|
	\end{align*}
	for all $z\in\mathbb{D}$.\vspace{2mm} 
	
	It is well-known that if $f(z_0)$ is a point of minimum modulus on the image of the circle $|z|=r$ and $\gamma=f^{-1}(\Gamma)$, where $\Gamma$ is the line segment from $0$ to $f(z_0)$, then 
	\begin{align*}
		|f(z)|\geq	|f(z_0)|\geq\int_{0}^{|z|}\frac{1}{(1+\xi^2)^\frac{c}{2}} d|\xi|.
	\end{align*}
	Thus all the desired inequalities are established. For the function \( f_{c, \lambda}(z) \) in the main result, the sharpness part can be shown easily, hence we omit the details.
\end{proof}
We will find the sharp bound of the pre-Schwarzian and Schwarzian norm for the function $f$ in the class $\mathcal{F}(c)$, under the assumption  that $f^{\prime\prime}(0)=0$. The following lemma will paly a key role to prove the result.
\begin{lemA}\cite{Carrasco-Hernández-AMP-2023}\label{LemA}
	If $\phi(z):\mathbb{D}\rightarrow\mathbb{D}$ be analytic function, then 
	\begin{align*}
		\frac{|\phi(z)|^2}{1-|\phi(z)|^2}\leq\frac{(\phi(0)+|z|)^2}{(1-|\phi(0)|)^2(1-|z|^2)|)}
	\end{align*}
\end{lemA}
Our result is sharp bound of the pre-Schwarzian norm for $f\in\mathcal{F}^0(c)$.
\begin{thm}\label{Th-2.3}
	For $c\in(0,3]$, let $f\in\mathcal{F}^0(c)$ be of the form \eqref{Eq-1.3}, then the pre-schwarzian norm satisfies the inequality
	\begin{align*}
		||Pf||\leq c.
	\end{align*}
	The inequality is sharp.
\end{thm}
\begin{proof}[\bf Proof of Theorem \ref{Th-2.3}]
	Since $\phi(z)=z\xi(z)$, with $|\xi(z)|<1$, then in \eqref{Eq-2.3} we obtain
	\begin{align*}
		\sup_{z\in\mathbb{D}}(1-|z|^2)\bigg|\frac{f^{\prime\prime}(z)}{f^{\prime}(z)}\bigg|&\leq\sup_{z\in\mathbb{D}}(1-|z|^2)\frac{c|z\xi(z)|}{1-|z|^2|\xi(z)|}\\&\leq c \sup_{0\leq r\leq 1}\frac{r(1-r^2)}{(1-r^2)}\\&=c.
	\end{align*}
	Thus, we have the required inequality $||Pf||\leq c$. \vspace{2mm}
	
	To show that the inequality is sharp, we consider the function $f_{c}^*$ given by
	\begin{align*}
		f_{c}^*(z)=	\int_{0}^{z}\frac{1}{(1-\xi^2)^\frac{c}{2}} d\xi.
	\end{align*}
	It can be easily shown that $||Pf_c^*||=c.$ This completes the proof.
\end{proof}
For the class $\mathcal{C}$ of convex functions, we obtain the following sharp result, derived from Theorem \ref{Th-2.3}.
\begin{cor}\label{Cor-2.4}
	Let $f\in\mathcal{C}:=\mathcal{F}^0(2)$ be of the form \eqref{Eq-1.3}, then the pre-schwarzian norm satisfies the inequality
	\begin{align*}
		||Pf||\leq 2.
	\end{align*}
	The inequality is sharp.
\end{cor}
\subsection*{Sharpness of Corollary \ref{Cor-2.4}}
For $c=2$, it follows from  that
\begin{align*}
	\frac{f_{2}^{\prime\prime}}{f_{2}^{\prime}}(z)=\frac{2z}{(1-z^2)}\;\;\mbox{and}\;\;Pf_2=\frac{2}{1-z^2}.
\end{align*}
A simple computation thus yields that
\begin{align*}
	||Pf_2||=\sup_{z\in\mathbb{D}}\left(1-z^2\right)|Pf_2|=\sup_{z\in\mathbb{D}}\left(1-|z|^2\right)\frac{2}{1-|z|^2}=2
\end{align*}
and we see the constant $2$ is best possible.\vspace{2mm}

Using the Schwarz lemma, we obtain a sharp bound for the Schwarzian derivative norm when \( f \in \mathcal{F}(c) \).
\begin{thm}\label{Th-2.4}
	For $c\in(0,3]$, let $f\in\mathcal{F}^0(c)$ be of the form \eqref{Eq-1.3}. Then the Schwarzian norm satisfies the inequality
	\begin{align*}
		||Sf||=\sup_{z\in\mathbb{D}}(1-|z|^2)^2|Sf(z)|\leq\frac{c(4-c)}{2}.
	\end{align*}
	The inequality is sharp. 
\end{thm}
	For the class of convex functions, we obtain the following immediate result from Theorem \ref{Th-2.4}, showing that the sharp bound of $||Sf||$ is $2$.
	\begin{cor}
		If $f\in \mathcal{C} := \mathcal{F}^0(2)$ be of the form \eqref{Eq-1.3} with $f^{\prime\prime}(0)=0$, then Schwarzian norm satisfies the inequality
		\begin{align*}
			||Sf||=\sup_{z\in\mathbb{D}}(1-|z|^2)^2|Sf(z)|\leq2.
		\end{align*}
		The estimate is sharp.
		
	\end{cor}
	\begin{proof}[\bf Proof of Theorem \ref{Th-2.4}]
		Let $f\in\mathcal{F}^0(c)$ be of the form \eqref{Eq-1.3}. Then from \eqref{Eq-2.3}, we have
		\begin{align*}
			\frac{f^{\prime\prime}(z)}{f^{\prime}(z)}=\frac{c\phi(z)}{(1-z\phi(z))}.
		\end{align*}
		A simple calculation gives that
		\begin{align*}
			Sf(z)=c\left(\frac{\phi^{\prime}(z)+\left(1-\frac{c}{2}\right)\phi^2(z)}{(1-z\phi(z))^2}\right).
		\end{align*}
		By using triangle inequality and Schwarz pick lemma, we obtain
		\begin{align}\label{Eq-2.6}
			(1-|z|^2)^2|Sf|&\leq c\bigg|\phi^{\prime}(z)+\left(1-\frac{c}{2}\right)\phi^2(z)\bigg|\frac{(1-|z|^2)^2}{|1-z\phi(z)|^2}\\&\nonumber=\frac{c(1-|z|^2)^2}{|1-z\phi(z)|^2}\left(\frac{1-|\phi(z)|^2}{1-|z|^2}+\left(1-\frac{c}{2}\right)|\phi(z)|^2\right).
		\end{align}
		We define the function $\Psi(z):\mathbb{D}\rightarrow\mathbb{D}$ such that
		\begin{align*}
			\Psi(z):=\frac{\bar{z}-\phi(z)}{1-z\phi(z)}.
		\end{align*}
		Since $\phi(\mathbb{D})\subseteq\mathbb{D}$ then $(1-|z|^2)(1-|z\phi(z)|^2)>0$, it follows that 
		\begin{align*}
			|\bar{z}-\phi(z)|^2<|1-z\phi(z)|^2.
		\end{align*}
		Thus, we conclude  that $|\Psi(z)|^2<1$. It is easy to see that
		\begin{align*}
			1-|\Psi(z)|^2=\frac{(1-|\phi(z)|^2)(1-|z|^2)}{|1-z\phi(z)|^2}
		\end{align*}
		and
		\begin{align}\label{Eq-2.7}
			\frac{(1-|z|^2)^2}{|1-z\phi(z)|^2}=\frac{(1-|\Psi(z)|^2)(1-|z|^2)}{(1-|\phi(z)|^2)}.
		\end{align}
		If we replace the expression \eqref{Eq-2.7} in \eqref{Eq-2.6}, then we have
		\begin{align}\label{Eq-2.8}
			(1-|z|^2)^2|Sf(z)|\leq c(1-|\Psi_1(z)|^2)\left(1+\left(1-\frac{c}{2}\right)\frac{|\phi(z)|^2(1-|z|^2)}{(1-|\phi(z)|^2)}\right).
		\end{align}
		Since $f^{\prime\prime}(0)=0$ implies that $\phi(0)=0$, using Lemma A, we obtain
		\begin{align}\label{Eq-2.9}
			\frac{|\phi(z)|^2}{1-|\phi(z)|^2}\leq\frac{|z|^2}{1-|z|^2}.
		\end{align}
		Using \eqref{Eq-2.9} in \eqref{Eq-2.8}, we obtain
		\begin{align*}
			(1-|z|^2)^2|Sf(z)|\leq c(1-|\Psi(z)|^2)\left(1+\left(1-\frac{c}{2}\right)|z|^2\right).
		\end{align*}
		Again, considering that $1-|\Psi(z)|^2\leq 1$, one can readily observe that
		\begin{align*}
			\sup_{z\in\mathbb{D}}(1-|z|^2)^2|Sh(z)|\leq c\left(1+\left(1-\frac{c}{2}\right)\right)=\frac{c(4-c)}{2}.\nonumber
		\end{align*}
		Thus the desired inequality is obtained. The sharpness of the inequality follows from the Example \ref{Example-2.1}.
	\end{proof}

\begin{exm}\label{Example-2.1}
	The family of parameterized functions defined as:
	\begin{align*}
		f_c (z)=\int_{0}^{z}\frac{1}{(1-\xi^2)^{c/2}} d\xi,\;\;\;\;\mbox{for}\;c\in(0,3]
	\end{align*}
	maximizes the Schwarzian norm defined as:
	\begin{align*}
		||Sf_{c}||=\sup_{z\in\mathbb{D}}(1-|z|^2)^2|Sf_{c}| 
	\end{align*}
	and from this, the sharpness of the inequality holds for $c>0$.\vspace{2mm} 
	
	\noindent Note that
	\begin{align*}
		\frac{f^{\prime\prime}_c}{f_c^{\prime}}(z)=\frac{c z}{1-z^2}\;\;\mbox{and}\;\;Sf_c=\frac{c}{(1-z^2)^2}\left[1+\left(1-\frac{c}{2}\right)|z|^2\right]
	\end{align*}
	which shows that
	\begin{align*}
		||Sf_c||&=\sup_{z\in\mathbb{D}}(1-|z|^2)^2|Sf_c|\\&=c\left(1+\left(1-\frac{c}{2}\right)\right)\\&=\frac{c(4-c)}{2
		}.
	\end{align*}
\end{exm}
Given that the value $|f^{\prime\prime}(0)|$ is not necessarily zero,  we provide a bound for the quantity $(1-|z|^2)^2|Sf(z)|$ where  $f\in\mathcal{F}(c)$.
\begin{thm}\label{Th-2.5}
	If $f\in\mathcal{F}(c)$, for all $z\in\mathbb{D}$ with $c\in(0,3]$ and 
	\begin{align*}
		\gamma=|\phi(0)|=\frac{|f^{\prime\prime}(0)|}{c},
	\end{align*}
	then 
	\begin{align*}
		(1-|z|^2)^2|Sf(z)|\leq c\left(1+\left(1-\frac{c}{2}\right)\frac{1+\gamma}{1-\gamma}\right).
	\end{align*}
\end{thm}
We have the following immediate result from Theorem \ref{Th-2.5} for the class $\mathcal{C}$ of convex functions.
\begin{cor}
	If $f\in \mathcal{C} := \mathcal{F}(2)$ be of the form \eqref{Eq-1.3}, then the inequality for $c=2$, with
	\begin{align*}
		\gamma=|\phi(0)|=\frac{|f^{\prime\prime}(0)|}{2}
	\end{align*}
the following inequality holds:
	\begin{align*}
		(1-|z|^2)^2|Sf(z)|\leq 2.
	\end{align*}
\end{cor}
\begin{proof}[\bf Proof of Theorem \ref{Th-2.5}]
	Let $\gamma=|\phi(0)|$. Applying the Lemma A, we  calculate
	\begin{align*}
		\frac{|\phi(z)|^2}{1-|\phi(z)|^2}\leq\frac{(\gamma+|z|)^2}{(1-\gamma^2)(1-|z|^2)}.
	\end{align*}
	If we substitute the above inequality into \eqref{Eq-2.6}, we get
	\begin{align*}
		(1-|z|^2)^2|Sf(z)|\leq c(1-|\Phi_1(z)|^2)\left(1+\left(1-\frac{c}{2}\right)\frac{(\gamma+|z|)^2}{(1-\gamma^2)}\right).
	\end{align*}
	From the fact that $|z|<1$ and  $1-|\Phi_1(z)|^2\leq1$, we can easily calculate 
	\begin{align*}
		(1-|z|^2)^2|Sf(z)|\leq c\left(1+\left(1-\frac{c}{2}\right)\frac{1+\gamma}{1-\gamma}\right)=2.
	\end{align*}
	This is the desired bound.
\end{proof}

\section{\bf Declarations}
\noindent\textbf{Conflict of interest:} The authors declare that there is no conflict  of interest regarding the publication of this paper.\vspace{2mm}

\noindent\textbf{Data availability statement:}  Data sharing not applicable to this article as no datasets were generated or analysed during the current study.\vspace{2mm}

\noindent {\bf Funding:} Not Applicable.
\end{document}